\documentclass[12pt]{amsart}
\usepackage{amssymb,latexsym} 
\usepackage{enumerate}

\makeatletter
\@namedef{subjclassname@2010}{%
  \textup{2010} Mathematics Subject Classification}
\makeatother

\newtheorem{thm}{Theorem}[section]

\newtheorem{prop}[thm]{Proposition}

\newtheorem{thmext}{Theorem}

\theoremstyle{definition}

\numberwithin{equation}{section}

\begin{document}
\baselineskip=17pt

\title[Friable integers and linear forms]{On the representation of friable integers by linear forms}

\author[A. Lachand]{Armand Lachand}
\address{
Institut \'Elie Cartan\\ Universit\'e de Lorraine\\ 
 B.P. 70239 \\
54506 Vand\oe uvre-l\`es-Nancy Cedex, France}
\email{armand.lachand@gmail.com}

\date{}

\begin{abstract}
Let $P^+(n)$ denote the largest prime of the integer $n$. Using the 
 nilpotent Hardy-Littlewood method developed by Green and Tao, we give an asymptotic formula for  
\begin{multline*}
\Psi_{F_1\cdots F_t}\left(\mathcal{K}\cap[-N,N]^d,N^{1/u}\right):=
\#\left\{\boldsymbol{n}\in \mathcal{K}\cap[-N,N]^d:\vphantom{P^+(F_1(\boldsymbol{n})\cdots F_t(\boldsymbol{n}))\leq N^{1/u}}\right.\\
\left.P^+(F_1(\boldsymbol{n})\cdots F_t(\boldsymbol{n}))\leq N^{1/u}\right\}
\end{multline*} where  $(F_1,\ldots,F_t)$ is a system of affine-linear forms of $\mathbf{Z}[X_1,\ldots,X_d]$ no two 
of which are affinely related and $\mathcal{K}$ is a convex body. This improves upon Balog, Blomer, Dartyge and Tenenbaum's work~\cite{BBDT12} 
in the case of product of linear forms.
\end{abstract}

\subjclass[2010]{Primary 11N25; Secondary 11N37}

\keywords{Friables integers, linear forms, Gowers norms}

\maketitle

\section{Introduction and statement of the result}

Given a real number $y>1$, an integer $n$ is said to be $y$-friable if its greatest prime factor, denoted by $P^+(n)$, satisfies $P^+(n)\leq y$ with the 
conventions $P^+(\pm1)=1$ and $P^+(0)=0$. Conversely, an integer $n$ is called $y$-sifted if its smallest prime factor, denoted by $P^-(n)$, satisfies $P^-(n)>y$ 
with the conventions $P^-(\pm1)=+\infty$ and $P^-(0)=0$. Due to the duality beetwen sifted integers and friable integers, such integers  occur in several places in number 
theory and their distribution has been intensively studied  (see \cite{HT93} and \cite{Gr08} for  survey articles related to integers without large prime factors). 
A theorem of Hildebrand~\cite{Hi86}, related to the number $\Psi(N,y)$ of $y$-friable integers  smaller than $N$, asserts that, for any $\varepsilon>0$ and uniformly in 
the domain \begin{equation}
N\geq 3\quad\text{and}\quad 1\leq u \leq \frac{\log N}{(\log\log N)^{5/3+\varepsilon}} 
,
\end{equation}
we have the asymptotic formula  
\begin{equation}\label{Hildebrand formula}
\Psi\left(N,N^{1/u}\right)=N\rho(u)\left(1+O\left(\frac{u\log(u+1)}{\log N}\right)\right)
\end{equation} 
where $\rho$ is the Dickman function, namely the unique solution  to  the delay differential equation
\begin{displaymath}
 \left\{\begin{array}{ll}\rho(u)=1&\text{if }0\leq u\leq 1,\\
        u\rho'(u)+\rho(u-1)=0&\text{if }u>1.\end{array}
                                 \right.
\end{displaymath}

 Given $F\in\mathbf{Z}[X_1,\ldots,X_d]$  and $\mathcal{K}\subset\mathbf{R}^d$, the study of the cardinality
\begin{displaymath}
\Psi_F
\left(\mathcal{K}
,y\right)
:=\#\left\{\boldsymbol{n}\in \mathcal{K}\cap\mathbf{Z}^d
: P^+(F(\boldsymbol{n}))\leq y\right\}
\end{displaymath}
is an interesting question. In particular, the factorization algorithm Number Field Sieve (NFS)\footnote{The interested reader may find a description of this algorithm 
in [\cite{CP05}, Chapter~6].} 
rests on the assumption that the cardinality
$
\Psi_F
\left(\mathcal{K},y\right)
$
is sufficiently large for some  small $y$, for $F\in\mathbf{Z}[X_1,X_2]$  and $\mathcal{K}\subset\mathbf{R}^2$ a sufficiently regular compact set.

Let 
 $F=F_1^{k_1}\cdots F_t^{k_t}$ be the decomposition of $F$ with $F_1,\ldots,F_t$ the distinct irreducible factors of $F$ and
   $d_1,\ldots,d_t$ their respective degrees with $d_1\geq \ldots\geq d_t\geq1$. If we assume  the events "$F_i(\boldsymbol{n})$ is $y$-friable" to be independent, 
 then  (\ref{Hildebrand formula}) leads to the following conjecture
\begin{equation}\label{conjecture friable binary form}
\Psi_F\left([0,N]^d,N^{1/u}\right)
\underset{N\rightarrow+\infty}{\sim} N^d\rho(d_1u)\cdots\rho(d_tu)
\end{equation}
for any fixed $u>0$.

When $d=2$, the author proved the validity of (\ref{conjecture friable binary form}) for  an irreducible cubic form $F$ or for $F=F_1F_2$ where $F_1$ is a linear form 
and $F_2$ is an irreducible quadratic form~\cite{Lacub, La15}. For general binary forms $F$, such a formula seems beyond reach but there exist some partial results  
 for estimating $\Psi_F\left([0,N]^2,N^{1/u}\right)$ when $u$ is sufficiently small.
 In \cite{BBDT12}, Balog, Blomer, Dartyge and Tenenbaum proved the existence of a constant $\alpha_F> 1/d_1$ such that, 
 for any $\varepsilon>0$ and uniformly for $N\geq2$, we have 
\begin{equation}\label{lower bound BBDT}
\Psi_F\left([0,N]^2,N^{1/\alpha_F+\varepsilon}\right)\gg_{\varepsilon} N^2.
\end{equation}

Let $d\geq2$ and $t\geq1$ be integers. In this paper, we focus on binary forms $F=F_1\cdots F_t$ where $F_1,\ldots ,F_t$ are 
some affine-linear forms in $\mathbf{Z}[X_1,\ldots, X_d]$. The cases $d=2$ and $t\in\{1,2\}$  can  be deduced from  results of \cite{FT91} related to 
the distribution of friable integers in arithmetic progressions. The case $d=2$ and $t=3$ was essentially considered by a succession of articles of 
various authors~(\cite{Br99,LS12,BG14,Dr13,Dr15,Haun}). 
In [\cite{Haun}, Corollary~1], Harper used the Hardy-Littlewood circle method to show the existence of $c>0$ such that, uniformly for $
N\geq2$ and $y\geq(\log N)^c$,  we have
\begin{align*}
\Psi_{X_1X_2(X_1+X_2)}\left(\mathcal{K}(N),y\right)
\underset{N\rightarrow+\infty}{\sim}\mathfrak{S}_0(\alpha,y)\mathfrak{S}_1(\alpha)\frac{\Psi\left(N,y\right)^3}{N}
\end{align*}
where $\mathcal{K}(N)=\left\{1\leq n_1,n_2\leq N:n_1+n_2\leq N\right\}$, $\alpha:=\alpha(N,y)$ denotes the unique real solution of the equation
\begin{displaymath}
\sum_{p\leq y}\frac{\log p}{p^{\alpha}-1}=\log N,
\end{displaymath}
\begin{displaymath}
\mathfrak{S}_0(\alpha,y):=\prod_{p\leq y}
\left(1+
\frac{(p-p^{\alpha})^3}{p(p-1)^2(p^{3\alpha-1}-1)}\right)
\prod_{p>y}\left(1-\frac{1}{(p-1)^2}\right)
\end{displaymath}
and\begin{displaymath}
\mathfrak{S}_1(\alpha):=\int_0^{1}\int_0^{1-t_1}\alpha^3
(t_1t_2(t_1+t_2))^{\alpha-1}\mathop{}\mathopen{}\mathrm{d} t_2\mathop{}\mathopen{}\mathrm{d} t_1.
\end{displaymath}

The celebrated work of Green, Tao and Ziegler~\cite{GT10,GT12a, GT12b,  GTZ12} provides a scheme - the so-called  nilpotent Hardy-Littlewood method - to 
get  asymptotic estimations of the average value
\begin{equation}\label{mean value h}
M_{F_1\cdots F_t}(\mathcal{K};h):=\sum_{\boldsymbol{n}\in\mathcal{K}\cap\mathbf{Z}^d}h(F_1(\boldsymbol{n}))\cdots h(F_t(\boldsymbol{n}))
\end{equation}
for  any system of affine-linear forms $F_1,\ldots, F_t\in\mathbf{Z}[X_1,\ldots, X_d]$ such that no two forms are affinely related and for any arithmetic function $h$
with a quasi-random behaviour. 
 In recent years, this approach has been applied successfully for several functions including the von Mangoldt functions $\Lambda$ (this gives a partial resolution 
 of the generalized 
 Hardy-Littlewood conjecture \cite{GT10}), the Liouville function $\lambda$ or the M\"obius function $\mu$~\cite{GT10}, the divisor function $\tau$~\cite{Ma12a},
 the function $r_G$ which counts the number of representations of a binary quadratic form $G$ \cite{Ma12b,Ma13} or, very recently, any multiplicative function that
 takes values in the unit  disk~\cite{FHun}.
 
In this work, we study how the nilpotent Hardy-Littlewood method may be  applied to get an asymptotic formula for (\ref{mean value h}) when $h=1_{S\left(N^{1/u}\right)}$ 
is the indicator function  of the  $N^{1/u}$-friable integers  for bounded $u\geq1$. Such a question is not covered 
by Frantzikinakis and Host work~\cite{FHun} since $h$ depends on $N$ in the present case. 
The main result is the following theorem.

\begin{thm}\label{main theorem n*1} Let $N,L,d,t$ and $u_0$ be some positive integers. Suppose that $F=(F_1,\ldots,F_t):\mathbf{Z}^d\rightarrow\mathbf{Z}^t$
 is a system of affine-linear forms such that any two forms $F_i$ and $F_j$ are affinely independent over $\mathbf{Q}$ and  the  non-constant coefficients of the  $F_i$ are bounded by $L$. 
 Then, for any convex body $\mathcal{K}\subset[-N,N]^d$ 
 such that 
 $F(\mathcal{K})\subset [0,N]^t$ and for any  $ u_1,\ldots, u_t\in[0, u_0]$, we have
\begin{align*}
\sum_{\boldsymbol{n}\in\mathcal{K}\cap\mathbf{Z}^d}1_{S\left(N^{1/u_1}\right)}(F_1(\boldsymbol{n}))\dots 1_{S\left(N^{1/u_t}\right)}(F_t(\boldsymbol{n}))=&
\mathrm{Vol}(\mathcal{K})\prod_{i=1}^t\rho(u_i)
+o(N^d) 
\end{align*}
where the implicit constant depends only on $t,d,L$ and  $u_0$ and $S(y)$ denotes the set of $y$-friable integers.
\end{thm}

As regards the result of Balog \textit{et al.}~\cite{BBDT12}, we get essentially two major improvements on their works in the case of linear forms :
\begin{itemize}
\item  Theorem \ref{main theorem n*1} gives an asymptotic equivalent which is consistent with the conjectural formula (\ref{conjecture friable binary form})
whereas  (\ref{lower bound BBDT}) only gives a lower bound,
\item when $t\geq 4$,  Formula (\ref{lower bound BBDT}) is  valid with  $\alpha_F=1+\frac{2}{t-2}$ while Theorem~\ref{main theorem n*1} shows that we
can  choose any positive real number for $\alpha_F$.
\end{itemize}

\textbf{Outline  and perspectives} In its primitive form, the nilpotent Hardy-Littlewood method is concerned with arithmetic functions $h$ which 
are equidistributed in residue classes of small moduli and supported on a set of integers with positive asymptotic density.
For such functions, the problem is reduced to show that $h$ is suitably Gowers-uniform to deduce asymptotics for $M_{F_1\cdots F_t}(\mathcal{K};h)$ 
(see the description of the method in Section~\ref{section descriptif}).

In many applications, the function $h$ may not satisfy the two previous conditions. 
The method developed in \cite{GT10, Ma12a, Ma12b}  to overcome this difficulties consists in two steps~:
\begin{itemize}
\item the decomposition of $h$ into a sum of functions which are equidistributed  in residue classes of small moduli ($W$-trick, see [\cite{GT10}, Section~5]),
\item the construction of a pseudorandom measure $\nu$ dominating $h$ in view to apply a transference principle (see [\cite{GT10}, Section~10]).
\end{itemize}

For bounded $u\geq1$, the set of $N^{1/u}$-friable integers  
has  positive density $\rho(u)$ and is well-behaved in arithmetic progressions 
of small common difference (see the work of Fouvry and Tenenbaum~\cite{FT91}). 
In particular, the problem may be directly handled by using the nilpotent Hardy-Litlewood method and showing 
that $h$ has small Gowers-uniformity norms. 
This may be viewed as an application of the impressive results of Matthiesen~\cite{Mattun}
related to the orthogonality beetwen multiplicative functions and  nilsequences.
In the Section~\ref{section preuve} of the present paper, we develop a more direct and simple approach to study
the linear correlations of the friable integers.

It would be interesting to prove Formula~(\ref{conjecture friable binary form}) for unbounded parameters $u$. 
In this case, the sequence of friable integers is too sparse to directly apply  Green-Tao-Ziegler's work. 
A major step to get this generalization would be to construct 
a pseudorandom majorant for $1_{S\left(N^{1/u}\right)}$. 

\section{A brief description of the nilpotent Hardy-Littlewood method}\label{section descriptif}

In this section, we recall  two important arguments of the nilpotent Hardy-Littlewood method.  
The generalized von Neumann theorem -- due to Gowers~\cite{Go01} and Green-Tao~\cite{GT10} -- reduces the estimation of 
$M_{F_1\ldots F_t}(\mathcal{K};h)$ defined in (\ref{mean value h}) to the study of the Gowers
uniformity norm  $\|h\|_{U^{t-1}[N]}$ (see [\cite{GT10}, Appendix~B] 
for a definition of Gowers norm).

\begin{thmext}[\cite{GT10}, Proposition~7.1] \label{neumannn}
Let $t,d,L\geq1$ be some integers. Suppose that $h_1,\ldots, h_t : [0,N]\to\mathbf{R}$ are functions bounded by $1$ and that  
$ F=(F_1,\ldots,F_t):\mathbf{Z}^d\rightarrow\mathbf{Z}^t$ is a system of affine-linear forms  whose non-constant coefficients are bounded by $L$ 
and such that any two forms $F_i$ and $F_j$ are affinely independent over $\mathbf{Q}$. Let  $\mathcal{K} \subset  [-N,N]^d$ be a convex body such that 
$F(\mathcal{K}) \subset  [0,N]^t$.  Suppose also that
\begin{equation*} \min_{1 \leq i \leq t}  \left\| h_i\right\|_{U^{t-1}[N]} \leq \delta
\end{equation*}
for some $\delta > 0$. Then we have
\begin{equation*}
\sum_{\boldsymbol{n} \in \mathcal{K}} \prod_{i=1}^t h_i(F_i(\boldsymbol{n})) =  
o_{\delta}(N^d) + \kappa(\delta)N^{d}
\end{equation*}
where $\kappa(\delta)\to0$ as $\delta\to 0$.
\end{thmext}
\begin{proof}
Let $(\boldsymbol{e}_1,\ldots,\boldsymbol{e}_d)$ be 
the canonical basis of $\mathbf{R}^d$ and fix $\boldsymbol{n}=n_1\boldsymbol{e}_1+ \cdot + n_d\boldsymbol{e}_d \in \mathcal{K}$. Then we have
$$\left|F_i(\mathbf{0})\right|\leq \sum_{i=1}^d \left|n_i\right|L+\left|F_i(\boldsymbol{n})\right|
\leq (dL+1)N$$
because $F(\boldsymbol{n})\in[0,N]^t$ and $\mathcal{K}\in[0,N]^d$.
With the definition (1.1) of the norm
$\|\cdot\|_N$ of \cite{GT10}, we therefore have $\|F\|_N\ll_{d,t}L$ and the 
Proposition~7.1 of \cite{GT10} can be used to get the result.
\end{proof}

The inverse theorem for the Gowers norms, proved by Green, Tao and Ziegler~\cite{GTZ12}, exhibits the link between linear correlations and polynomial nilsequences. 
The reader may refer to \cite{GT12b} for definitions and properties of filtered nilmanifolds and polynomial nilsequences. 
\begin{thmext}[\cite{GTZ12}, Theorem~1.3]\label{Gowers inverse}
Let $s\geq0$ be an integer and let $\delta\in]0,1]$.   Then there exists a finite collection $\mathcal{M}_{s,\delta}$ of $s$-step nilmanifolds
$G/\Gamma
$, each equipped with some smooth Riemannian metric   $d_{G/\Gamma}$, as well as positive constants $C(s,\delta)$ and $c(s,\delta)$ with the following property. 
Whenever $N\geq1$ and $h : [0,N]\cap\mathbf{Z} \rightarrow [-1,1]$ is a  function such that
\begin{displaymath} \left\|h \right\|_{U^{s+1}[N]} 
\geq \delta,
\end{displaymath}
there exists a filtered nilmanifold $G/\Gamma \in \mathcal{M}_{s,\delta}$, a 
function $F:G/\Gamma\rightarrow\mathbf{C}$ bounded in magnitude by $1$ and with Lipschitz constant at most $C(s,\delta)$ 
with respect to the metric $d_{G/\Gamma}$ and a polynomial nilsequence $g:\mathbf{Z}\rightarrow G$   such that
\begin{displaymath} \left|
\sum_{0\leq n \leq N} h(n) F\left(g(n) \Gamma\right)\right| 
\geq c(s,\delta) N.
\end{displaymath}
\end{thmext}

We describe now the application of the Green-Tao method to the functions $1_{S\left(N^{1/u_i}\right)}$. For any parameter of friability $N^{1/u_i}$, 
we consider the balanced function 
\begin{displaymath}
\begin{array}{cccc}h_i:&\mathbf{N}&\rightarrow&[-1,1]\\
&n&\mapsto&1_{S\left(N^{1/u_i}\right)}(n)-\rho(u_i).
\end{array}
 \end{displaymath}
By writing  $1_{S\left(N^{1/u_i}\right)}(n)=h_i(n)+\rho(u_i)$ and using the bound $\rho(u_i)\leq1$, it follows that
\begin{multline*}
  \left|\Psi_{F_1\cdots F_t}\left(\mathcal{K}
 ,N^{1/u}\right)
 -\mathrm{Vol}(\mathcal{K})\prod_{i=1}^t\rho(u_i)\right|
\leq\sum_{\substack{I\subset\{1,\ldots, t\}\\I\neq\emptyset}} \left|\sum_{\boldsymbol{n} \in\mathcal{K}\cap\mathbf{Z}^d
}\prod_{i\in I}h_i(F_i(\boldsymbol{n}))\right|\\+O_d\left(N^{d-1}\right).
\end{multline*}

In view of the inverse theorem, the problem is reduced to prove that, for any $i\in\{1,\ldots, t\}$, the function $h_i$ does not correlate with nilsequences,
namely that the upper bound
\begin{displaymath}
\sum_{n\leq N}h_i(n)F(g(n)\Gamma) =o\left(N\right)
\end{displaymath}
holds for any $(t-2)$-steps nilsequences $F(g(n)\Gamma)$.

\section{Non-correlation with nilsequences}\label{section preuve}
Let  $s\geq0, u_0\geq1$ be some integers and
let  $(G/\Gamma,G_{\mathbf{N}})$ be a filtered nilmanifold of degree $s$. In this section, we show that for any $1$-bounded 
 Lipschitz function $F:G/\Gamma\rightarrow\mathbf{C}$,  any polynomial nilsequence $g:\mathbf{Z}\rightarrow G$  adapted 
 to  $G_{\mathbf{N}}$, $1\leq u\leq u_0$ and $N\geq1$, we have
 \begin{equation}\label{formula to show n*1}
 \sum_{n\leq N}h(n)F(g(n)\Gamma) =o\left(N\left(1+\|F\|_{\textrm{Lip},d_{G/\Gamma}
 }\right)\right)
 \end{equation}
where $h(n):=1_{S\left(N^{1/u}\right)}(n)-\rho(u)$
  and the implicit term  $o(\cdot)$  only depends on $G/\Gamma$ and $u_0$. In view of Theorems~\ref{neumannn} 
  and \ref{Gowers inverse},    this will imply Theorem~\ref{main theorem n*1}.

In \cite{Mattun}, Mathiesen develop a method to bound the correlations  of a multiplicative function
with polynomial nisequences, under some density and growth conditions and some hypothesis of control of the second moment. 
Its approach mix the Montgomery-Vaughan method~\cite{MV77}, the factorisation theorem for polynomial sequences from Green-Tao~\cite{GT12b}
and the fact that the $W$-tricked von Mangoldt function is orthogonal to nilsequences~\cite{GT10}.
 Its main result~[\cite{Mattun}, Theorem 5.1] may be applied directly to 
the multiplicative function $1_{S\left(N^{1/u}\right)}(n)$ to get (\ref{formula to show n*1}), once we have checked it satisfies the assumptions required. 
In the case of the indicator of friable integers and for any $E\geq1$,
the various hypothesis which defined the set $\mathcal{F}_1(E)$ of \cite{Mattun} 
 can be essentially deduced from the estimation 
 \begin{displaymath}\sum_{\substack{n\leq N\\ n\equiv a\pmod{q}}}1_{S\left(N^{1/u}\right)}(n)\underset{N\to\infty}{\sim} \frac{N}{q}\rho(u) 
 \end{displaymath}
 which holds uniformly for $1\leq a,q\leq (\log N)^{E}$ (see \cite{FT91}).

In the rest of this paper, we give a direct and simple method to establish (\ref{formula to show n*1}), 
with a different focus from
\cite{Mattun}.
The starting point 
is the  M\"obius inversion formula in the following form
\begin{displaymath}1_{S\left(N^{1/u}\right)}(n)=\sum_{\substack{P^-(k)>N^{1/u}}}\mu(k)1_{k|n}.
\end{displaymath} 
We  approximate the indicator $1_{k|n}$ by its mean value $\frac{1}{k}$ for $k\leq N^{1-\tau}$ where the parameter $\tau=o(1)\in]1/\log N,1[$ 
will be chosen later. One can write\begin{align*}
\sum_{1\leq n\leq N}
\left(1_{S\left(N^{1/u}\right)}(n)-\rho(u)\right)F(g(n)\Gamma)
=\Sigma_1(F,g)+\Sigma_2(F,g)
\end{align*}
where
\begin{align*}
\Sigma_1(F,g):=
\sum_{1\leq n\leq N}h_{\tau}(n)
F(g(n)\Gamma)\text{ with }h_{\tau}(n)=\sum_{\substack{k\leq N^{1-\tau}\\P^-(k)>N^{1/u}}}\mu(k)\left(1_{k|n}-\frac{1}{k}\right)
\end{align*}
and
\begin{align*}
\Sigma_2(F,g)
:=\sum_{1\leq n\leq N}\left(\sum_{\substack{k> N^{1-\tau}\\P^-(k)>N^{1/u}}}\mu(k)1_{k|n}+
\sum_{\substack{k\leq N^{1-\tau}\\P^-(k)>N^{1/u}}}\frac{\mu(k)}{k}-\rho(u)\right)F(g(n)\Gamma).
\end{align*}
In the definition of the function $h_{\tau}$,  the summation is restricted over the divisors $k\leq N^{1-\tau}$ 
since the contribution from the interval
$N^{1-\tau}<k\leq N$ is negligible (see (\ref{borne sup sigma2}) below).

First, we focus on $\Sigma_{2}(F,g)$. 
In view of the following series of estimations, valid whenever $\tau u<1$,
\begin{align}
\nonumber\sum_{\substack{N^{1-\tau}<k\leq N\\P^-(k)>N^{1/u}}}\frac{\mu^2(k)}{k}&\ll
\sum_{j\geq1}\sum_{N^{1/u}<p_2<\cdots<p_j\leq N}\frac{1}{p_2\cdots p_j} 
\sum_{\max\left(N^{1/u},\frac{N^{1-\tau}}{p_2\cdots p_j}\right)\leq p_1\leq \frac{N}{p_2\cdots p_j}}\frac{1}{p_1}\\
\nonumber&\ll\tau u \sum_{j\geq1}\frac{1}{(j-1)!}\left(\sum_{N^{1/u}<p\leq N}\frac{1}{p}\right)^{j-1}\\
\label{mu2}&\ll \tau u\sum_{j\geq1}\frac{1}{(j-1)!}\left(\log(u)+O\left(1\right)\right)^{j-1}\ll\tau u^2,
\end{align}
we have the upper bound
\begin{align}\label{Sigma23}
\sum_{1\leq n\leq N}\sum_{\substack{k> N^{1-\tau}\\P^-(k)>N^{1/u}}}\mu^2(k)1_{k|n}\ll \tau u^2N.
\end{align}
On the other hand, one can handle the sum over $k\leq N^{1-\tau}$ in $\Sigma_2(F,g)$ by using [\cite{LT15}, Formula~(1.5)] which states that the formula
\begin{equation}\sum_{\substack{k\leq N\\P^-(k)>N^{1/u}}}\frac{\mu(k)}{k}=\rho(u)\left(1+O\left(\frac{u\log(u+1)}{\log N}\right)\right)\label{Sigma21} \end{equation} 
holds for any $\varepsilon>0$ and uniformly for $x\geq2$ and  $1\leq u\leq (\log x)^{3/8-\varepsilon}$.
Finally, 
 (\ref{Sigma23}) and (\ref{Sigma21}) yield
  that
\begin{align}
\Sigma_2(F,g)
\ll&uN\left(\tau u +\frac{\rho(u)\log(u+1)
}{\log N}\right)\label{borne sup sigma2}.\end{align}

In view of the foregoing, it remains to obtain  an upper bound for $\Sigma_1(F,g)$. This is the subject of the following proposition.
\begin{prop}\label{estimation Sigma1}Let $m,s\geq1$ be some integers and let $A>0$ be a real number. There exists a constant  $c(m,s,A)>0$
with the following property. Whenever $Q,N\geq2$ are integers, $\tau\in]0,1/2[$ and $u\geq1$ are such that
$\min(N^{\tau},N^{1/u})\geq (\log N)^{c(m,s,A)}$, $(G/\Gamma,G_{\mathbf{N}})$  is a filtered nilmanifold of degree $s$ and dimension $m$,  $\mathcal{X}$ is a
$Q$-rational Mal'cev 
basis\footnote{The notion of $Q$-rational Mal'cev basis is introduced in [\cite{GT12b}, Definitions~2.1 and 2.4]
as a specific basis of the Lie algebra $\mathfrak{g}$ of $G$.} of $(G/\Gamma,G_{\mathbf{N}})$, $g:\mathbf{Z}\rightarrow G/\Gamma$ 
is a polynomial nilsequence adapted to $G_{\mathbf{N}}$ and $F:G/\Gamma\rightarrow[-1,1]$ is a Lipschitz function, then we have  
\begin{equation}\label{formula Sigma1}
\sum_{1\leq n\leq N}h_{\tau}(n)F(g(n)\Gamma)\leq N Q^{c(m,s,A)}\left(1+\|F\|_{\textrm{Lip},\mathcal{X}}\right)2^u(\log N)^{-A}.
\end{equation}
\end{prop}
Recall that the smooth Riemannian metric $d_{G/\Gamma}$ of  Proposition~\ref{Gowers inverse} is equivalent to the metric $d_{\mathcal{X}}$ (see the 4th 
footnote and Definition~2.2 of \cite{GT12b}).  With the choice  $\tau=\frac{(\log\log N)^{1+\varepsilon}}{\log N}$, it follows from the estimations 
(\ref{borne sup sigma2}) and (\ref{formula Sigma1}) that 
the upper bound 
\begin{displaymath} \sum_{n\leq N}h(n)F(g(n)\Gamma) =o\left(N\rho(u)\left(1+\|F\|_{\textrm{Lip},d_{G/\Gamma}
 }\right)\right)\end{displaymath}
 holds for any $\varepsilon>0$  and uniformly for $1\leq u\leq (\log\log N)^{1-\varepsilon}$. 
 This implies (\ref{formula to show n*1}) since $1\leq u\leq u_0$ is contained in this region for any  $u_0$ which does not depend on $N$.

The rest of the article is devoted to  the proof of Proposition~\ref{estimation Sigma1}. 
The argument follows essentially the proofs of [\cite{GT12a},Theorem~1.1] and \cite{Ma12a}, Theorem~9.1] and we only outline the major differences.
A key point in the proof consists in reducing the problem to establish  the formula (\ref{formula Sigma1}) 
in the case of totally equidistributed polynomial nilsequence $g$, i.e. such that $|P|^{-1}\sum_{n\in P}F(g(n)\Gamma)$ tends to $\int_{G/\Gamma}F$ as $P$ is a 
subprogression such that $|P|\rightarrow +\infty$.

After this reduction, it will be possible to use the following analogue of [\cite{GT12a}, Proposition~2.1] and 
[\cite{Ma12a}, Proposition~9.2].
\begin{prop}\label{non correlation equidistributed}
Let $m,s$ be some positive integers. There exist some constants  $c_0(m,s),c_1(m,s)>0$ with the following property. 
Whenever $Q\geq2$, $N\geq2$  and  $\delta\in]0,1/2[$  such that $\delta^{-c_0(m,s)}\leq N^{\tau}$, $P\subset\{1,\ldots,N\}$ is an arithmetic progression 
of size at least $N/Q$, $(G/\Gamma,G_{\mathbf{N}})$  is a filtered nilmanifold of degree $s$ and dimension $m$,  $\mathcal{X}$ is a  $Q$-rational Mal'cev basis
of $(G/\Gamma,G_{\mathbf{N}})$, $g:\mathbf{Z}\rightarrow G/\Gamma$ is a polynomial 
and $\delta$-totally equidistributed nilsequence\footnote{A sequence $\left(g(n)\Gamma\right)_{n\in\{1,\ldots,N\}}$ is $\delta$-totally equidistributed if we have 
$$\left|\frac{1}{|P|}\sum_{n\in P}F(g(n)\Gamma)\right|\leq \delta\|F\|$$ for all Lipschitz function $F:G/\Gamma\rightarrow\mathbf{C}$ with 
$\int_{G/\Gamma}F=0$ and all arithmetic progressions $P\subset\{1,\ldots,N\}$ of size at least $\delta N$.} adapted to $G_{\mathbf{N}}$ and 
$F:G/\Gamma\rightarrow[-1,1]$ is a Lipschitz function such that $\int_{G/\Gamma}F=0$,  we have  
\begin{equation*}
 \left|\sum_{n\leq N}h_{\tau}(n)1_P(n)F(g(n)\Gamma)\right|\ll \delta^{c_1(m,s)}\|F\|_{\textrm{Lip},\mathcal{X}}QN \left(2^u+\log N\right).
\end{equation*}
\end{prop}
\begin{proof}[Proof that Proposition~\ref{non correlation equidistributed} implies Proposition~\ref{estimation Sigma1}]
Following some ideas of \cite{GT12a}, we can assume, without loss of generality, that  
$\|F\|_{\textrm{Lip},\mathcal{X}}=1$ and $Q\leq\log N$. Let $B>0$ be a parameter to be specified at the end of the proof. Applying Theorem~1.19 of \cite{GT12b},
there exists an integer $M$ satisfying $\log N\leq M\leq (\log N)^{c(m,s,B)}$
such that we can write the decomposition  $g=\varepsilon g'\gamma$ where \begin{enumerate}
 \item $\varepsilon\in\textrm{poly}(\mathbf{Z},G_{\mathbf{N}})$ is $(M,N)$-smooth (see [\cite{GT12b}, Definition~1.18]),
\item $g'\in\textrm{poly}(\mathbf{Z},G_{\mathbf{N}})$ takes values in a rational subgroup $G'\subseteq G$ with Mal'cev basis $\mathcal{X}'$ and 
$(g'(n))_{n\leq N}$ is $M^{-B}$-totally 
equidistributed in $G'/(G'\cap\Gamma)$ for the metric $d_{\mathcal{X}}$ (see [\cite{GT12b}, Definition~1.10]),
\item $\gamma\in\textrm{poly}(\mathbf{Z},G_{\mathbf{N}})$ is periodic of period $q\leq M$ and $\gamma(n)$ is $M$-rational  for any $n\in\mathbf{Z}$  
(see [\cite{GT12b}, Definition~1.17]).
\end{enumerate} 

Next, we reproduce  the arguments of Green and Tao based on partitioning and pigeonholing  and we use the properties of periodicity and smoothness 
of $\gamma$ and  $\varepsilon$. 
In this way,  the problem is reduced to show that
\begin{equation}\label{intermediary estimat Sigma1}
\left|\sum_{1\leq n\leq N}h_{\tau}(n)1_{P}(n)F'(g'(n)\Gamma')\right|\ll 2^{u}N/(M^2(\log N)^{2A})
\end{equation}
where $P$ is a subprogression such that 
$|P|\geq \frac{N}{2M^2(\log N)^A}$,  $(G'/\Gamma',G'_{\mathbf{N}})$ is a 
$m$-dimensional nilmanifold of degree $s$ with $M^{C_1(m,s)}$-rational Mal'cev basis 
$\mathcal{X}'$, $F':G'/\Gamma'\rightarrow[-1,1]$ is a Lipschitz function such that $\|F'\|_{\textrm{Lip},\mathcal{X}'}\leq M^{C_1(m,s)}$
and $g'\in\textrm{poly}(\mathbf{Z},G'_{\mathbf{N}})$ is $M^{-C_2(m,s)B+C_1(m,s)}$-totally equidistributed, 
for some constants $C_1(m,s),C_2(m,s)>0$.

If we suppose that $\int_{G'/\Gamma'}F'=0$, then we can apply Proposition~\ref{non correlation equidistributed} 
to the sequence $g'$,  
with   $M^{C_1(m,s)}$ (resp. $M^{-C_2(m,s)B+C_1(m,s)}$) as parameter of rationality (resp. totally equidistribution).
Taking $B$, 
$C_1(m,s)$ and $c(m,s,A)$ sufficiently large, the hypothesis on the size of $P$ and $\delta $ are satisfied and we get 
(\ref{intermediary estimat Sigma1}).

We can reduce to this last case by writing $F'=(F'-\int_{G'/\Gamma'}F')+\int_{G'/\Gamma'}F'$. 
Indeed, we can observe that $\int_{G'/\Gamma'}F'$ is bounded by $1$ and,  since the common difference $q$ of $P$ 
satisfies  $q< N^{1/u}$, then we get some multiplicative independence, when $P^-(k)>N^{1/u}$ :
$$\left|\sum_{1\leq n\leq N}\left(1_{k|n}-\frac{1}{k}\right)1_{P}(n)\right|\leq1. $$ We deduce the major arc estimate
\begin{align*}
\left|\sum_{1\leq n\leq N}h_{\tau}(n)1_{P}(n)\int_{G'/\Gamma'}F'\right|
&\leq\sum_{\substack{k\leq N^{1-\tau}\\P^-(k)>N^{1/u}}}\left|
\sum_{1\leq n\leq N}\left(1_{k|n}-\frac{1}{k}\right)1_{P}(n)\right|\\
&\leq \left|\left\{k\leq N^{1-\tau}:P^-(k)>N^{1/u}\right\}\right|\\
&\ll u\frac{N^{1-\tau}}{\log N}
\end{align*}
which implies (\ref{intermediary estimat Sigma1}) under the condition $N^{\tau}\geq (\log N)^{c(m,s,A)}$.
\end{proof}

\begin{proof}[Proof of Proposition~\ref{non correlation equidistributed}] We essentially follow the proof of Proposition~9.2 of \cite{Ma12a}
and we suppose that $\|F\|_{\textrm{Lip},\mathcal{X}}=1$ and $Q\leq \delta^{-c_1(m,s)}$. 
 For $\mathcal{T}\in]0,1/2[$ and $j\geq1$, we define $S_j(\mathcal{T})$ as the set of the integers $k$ satisfying
\begin{displaymath}
 \left|\sum_{2^j/k<n\leq 2^{j+1}/k}1_P(kn)F(g(kn)\Gamma)\right|> \mathcal{T}\frac{2^j}{k}.
\end{displaymath}
From the estimation 
$$\sum_{\substack{k\leq N\\ P^-(k)>N^{1/u}}}\frac{\mu^2(k)}{k}
%\sum_{j\geq1}\sum_{N^{1/u}<p_1<\cdots<p_j\leq N}\frac{1}{p_1\cdots p_j} 
\ll \sum_{j\geq1}\frac{1}{j!}\left(\sum_{N^{1/u}<p\leq N}\frac{1}{p}\right)^{j}
\ll  u
$$%(\ref{mu2}), (\ref{Sigma23}) 
and the trivial bound $\left|\left\{k|n: P^-(k)>N^{1/u}\right\}\right|\leq 2^u$ valid whenever $n\leq N$,
we can see that $h_{\tau}(n)\ll 2^u$. 
It follows that
\begin{align*}
\left|\sum_{n\leq N^{1-\tau/2}}h_{\tau}(n)1_P(n)F(g(n)\Gamma)\right|\ll N^{1-\tau/2}2^u
\end{align*}
and therefore we concentrate on the integers  $n> N^{1-\tau/2}$. 

Since the nilsequence $(g(n)\Gamma)_{n\in\{1,\ldots,N\}}$ is $\delta$-totally equidistributed, 
the contribution from the part $\sum_k\frac{\mu(k)}{k}$ of $h_{\tau}$ can be handled 
by observing that  we have
\[
\sum_{\substack{k\leq N^{1-\tau}\\ P^-(k)>N^{1/u}}}\frac{\mu^2(k)}{k}
\left|\sum_{N^{1-\tau/2}\leq n \leq N}1_{P}(n)F(g(n)\Gamma)\right|\ll u\delta N.
\]

For the remaining terms $\sum_k\mu(k)1_{k|n}$ of $h_{\tau}$, we follow
the proof of Proposition~9.2 of \cite{Ma12a}.
We  make a dyadic splitting over the variables $k$ and $n$ and we drop off the condition $P^-(k)>N^{1/u}$ :   
\begin{align*}
 &\sum_{\substack{k\leq N^{1-\tau}
 }}\left|\sum_{N^{1-\tau/2}/k<n\leq N/k}1_P(kn)F(g(kn)\Gamma)\right|\\
 \ll& \sum_{2^i\leq N^{1-\tau}}\sum_{\frac{N^{1-\tau/2}}{2}\leq 2^j\leq N}2^j\left(\sum_{\substack{2^{i}\leq k< 2^{i+1}
}}\frac{\mathcal{T}}{k}
+\sum_{\substack{2^{i}\leq k< 2^{i+1}
\\k\in S_j(\mathcal{T})}}\frac{1}{k}\right)\\
\ll& \sum_{\frac{N^{1-\tau/2}}{2}\leq 2^j\leq N}2^j\left(
\mathcal{T}\log N+\sum_{2^i\leq N^{1-\tau}}\frac{1}{2^i}\#\left( S_j(\mathcal{T})\cap\left[2^{i},2^{i+1}\right]\right)\right).
\end{align*}

Put $\mathcal{T}:=\delta^{c_1(m,s)}\leq Q^{-1}$ for a constant $c_1(m,s)>0$ sufficiently small. In the previous sum, 
the contribution of the
 range $\frac{N^{1-\tau/2}}{2}\leq 2^j\leq \mathcal{T} N$ is negligible and may be bounded by the trivial inequality.
 
The rest of the proof consists in  showing that, if $K\leq N^{1-\tau}$, then we have 
\begin{equation}\label{estimation s delta c}
\# \left(S_j(\mathcal{T})\cap[K,2K]\right)\leq \mathcal{T}K
\end{equation}
whenever $\mathcal{T}N\leq 2^j\leq N$.

The estimate (\ref{estimation s delta c}) is the analogue of [\cite{Ma12a}, Lemma~9.3]  under the constraint $K\leq N^{1-\tau}$ rather than $K\leq N^{1/2}$ 
and in the special case $\overline{W}=1$ and $b=0$.
To achieve this, we follow the discussion of
Type I case of [\cite{GT12a}, Part~3] and we suppose for contradiction that (\ref{estimation s delta c}) does not hold for some  $K\leq N^{1-\tau}$ and 
$\mathcal{T}N
\leq 2^j\leq N$. By reproducing  their arguments, we observe the existence 
of a non-trivial horizontal character $\psi$ with magnitude $0<|\psi|\leq \mathcal{T}^{-c_2(m,s)}$ such that, for any $r\geq1$ and for at least  
$\mathcal{T}^{c_2(m,s)}K$ values of $k$, we have
\begin{equation*}
\|\partial^r(\psi\circ g_k)(0)\|_{\mathbf{R}/\mathbf{Z}}\leq\mathcal{T}^{-c_2(m,s)}\left(K/2^j\right)^{r}
\end{equation*} 
where $g_k(n)=g(kn)$,
which is the analogue of the formula~(3.7) of \cite{GT12a}.

By Lemma~3.2 and 3.3 of \cite{GT12a} -- consequences of Waring's theorem -- it follows that 
there exists an integer $q\ll_s1$ and at least $\mathcal{T}^{c_3(m,s)}K^r$ integers $l\leq 10^sK^r$ such that
\begin{equation*}
\|ql\beta_r\|_{\mathbf{R}/\mathbf{Z}}\leq \mathcal{T}^{-c_3(m,s)}(K/2^j)^r
\end{equation*}
where the $\beta_r$'s are  defined by \begin{equation}\label{definition beta}
\psi\circ g(n)=\beta_s n^s+\cdots+\beta_0.\end{equation}

To deduce some diophantine information about the $\beta_r$'s, we  invoke Lemma~3.2 of \cite{GT12b} 
in an analogous way as \cite{GT12a} after checking that the hypothesis are satisfied. It suffices to see that 
$r\geq1$ and $\frac{\mathcal{T}^{2c_3(m,s)}}{10^s}\gg N^{-\tau}\geq\left(\frac{K}{2^j}\right)^r$ if the constant $c_1(m,s)$ is chosen sufficiently small. 
It results that there exists  $q'\leq \mathcal{T}^{-c_4(m,s)}$ such that 
\begin{displaymath}
\|q'\beta_r\|_{\mathbf{R}/\mathbf{Z}}\leq\mathcal{T}^{-c_4(m,s)}2^{-rj}
\end{displaymath}
for any integer $r\geq1$. 
By  the definition~(\ref{definition beta}), we get the existence of $c_5(m,s)>0$ sufficiently 
large such that $q'\leq \mathcal{T}^{-c_5(m,s)}$ and 
\begin{equation}\label{derniere estimation diophante}
\|q'(\psi\circ g)(n)\|_{\mathbf{R}/\mathbf{Z}}\leq 1/10
\end{equation}for any $n\leq\mathcal{T}^{ c_5(m,s)}2^j$.

Let $\eta:\mathbf{R}/\mathbf{Z}\longrightarrow[-1,1]$ be a Lipschitz function of norm $O(1)$, mean value zero, and equal to $1$ 
on $[-1/10,1/10]$ so that
\begin{displaymath}
\int_{G/\Gamma}\eta\circ (q'\psi)=0\qquad\text{and}\qquad \|\eta\circ( q'\psi)\|_{\textrm{Lip},\mathcal{X}}\leq \mathcal{T}^{-c_5(m,s)}.
\end{displaymath}
It follows from  (\ref{derniere estimation diophante}) that we have
\begin{displaymath}
\left|\sum_{n\leq\mathcal{T}^{ c_5(m,s)}2^j}\eta (q'\psi(g(n)\Gamma))\right|\geq \mathcal{T}^{ c_5(m,s)}2^j%-1
>\delta
\|\eta\circ( q'\psi)\|_{\textrm{Lip},\mathcal{X}}\mathcal{T}^{c_5(m,s)}2^j
\end{displaymath}
whenever $c_1(m,s)$ is sufficiently small. This contradicts the hypothesis that   
$(g(n))_{n\leq N}$ is $\delta$-totally  equidistributed,  the set  of integers less than   $\mathcal{T}^{ c_5(m,s)}2^j$ 
being an  arithmetic progression of size at least $\delta N$ whenever $c_1(m,s)$ is sufficiently small since $2^j\geq\mathcal{T}
N$.
\end{proof}

\subsection*{Acknowledgements} The author would like to thank Trevor Wooley for his suggestion to study this method, 
R\'egis de la Bret\`eche, Fran\c cois Hennecart and Anne de Roton 
for their interest for this work, and  his Ph.D. advisor C\'ecile Dartyge for her continuous support.
The major part of this work were completed while the first author was a Ph.D. student at Universit\'e de Lorraine. He put 
the finishing touch while he was a postdoctoral fellow at Aix-Marseille Universit\'e.

\end{document}